\documentclass[12pt,oneside,reqno]{amsart}       
\usepackage{enumerate}
\usepackage{tikz}
\usepackage{a4,amsmath,amssymb,amsthm,amscd,enumerate,amscd,}
\usepackage{tikz}
\usepackage{amsfonts,eufrak}
\usepackage[utf8]{inputenc}
\usepackage{amsmath}
\usepackage{amsfonts}
\usepackage{graphicx}
\usepackage{amssymb}
\usepackage{amsthm}

\usepackage[colorlinks=true, urlcolor=Blue, pdfborder={0 0 0}]{hyperref}

\usepackage{graphicx}
\usepackage{forest}
\usepackage{tikz-qtree}

\makeatletter
\@namedef{subjclassname@2020}{%
	\textup{2020} Mathematics Subject Classification}
\makeatother
\theoremstyle{definition}
\newtheorem{theorem}{Theorem}[section]

\newtheorem{proposition}[theorem]{Proposition}

\theoremstyle{definition}
\theoremstyle{definition}

\newtheorem{remark}[theorem]{Remark}
\newtheorem{example}[theorem]{Example}
\usepackage{indentfirst}

\begin{document}
	\baselineskip=16.5pt
	\title[]{involutions on the product of quaternionic projective space  and  sphere}
	\author[Dimpi and Hemant Kumar Singh]{ Dimpi and Hemant Kumar Singh}
	\address{ Dimpi \newline 
		\indent Department of Mathematics\indent \newline\indent University of Delhi\newline\indent 
		Delhi -- 110007, India.}
	\email{dimpipaul2@gmail.com}
	\address{  Hemant Kumar Singh\newline\indent 
		Department of Mathematics\newline\indent University of Delhi\newline\indent 
		Delhi -- 110007, India.}
	\email{hemantksingh@maths.du.ac.in}

	\date{}
	\thanks{ The first author of the paper is  supported by SRF of UGC, New Delhi, with  reference no.: 201610039267.}
	\begin{abstract} 
		
		Let $G=\mathbb{Z}_2$ act on a  finite  CW-complex $X$ having mod $2$  cohomology isomorphic to  the product of quaternionic projective space and sphere $\mathbb{H}P^n\times \mathbb{S}^m, n,m\geq 1.$ This paper is concerned with the connected  fixed point sets and the orbit spaces of free involutions  on $X.$

	\end{abstract}
	

	\maketitle
	\section {Introduction}
	\noindent  Let $(G,X)$ be a transformation group, where $G$ is compact Lie group  and $X$ is finite CW-complex,  with the fixed point set $F.$ The study of the cohomological structure of the fixed point set and orbit space has been an interesting problem in transformation groups.  Smith \cite{s} proved that the fixed point sets of $G=\mathbb{Z}_p,$ $p$ a prime, on a finite  dimensional polyhedron $X$ having mod $p$ cohomology $n$-sphere are  mod $p$  cohomology $r$-sphere, where $-1\leq r \leq n.$ He has also proved that if $G=\mathbb{Z}_2$ acts effectively on the real projective space then the fixed point set is either empty or it has two components having mod $2$ cohomology real projective space \cite{smith}.  Bredon \cite{ Bredon} generalizes this result for $G=\mathbb{Z}_p,~ p$ a prime, actions on  cohomology projective spaces.   Bredon \cite{Bredon} also proved   that if a finitistic space $X$  satisfies poincar\'{e} duality with respect to  \v{C}ech cohomology with $\mathbb{Z}_p$-coefficients then each component of the fixed point set also satisfies poincar\'{e}  duality.  Puppe \cite{puppe} proved Bredon's  conjecture which states that  if $X$ is totally nonhomologous to zero in $X_G$ (Borel space) then the number of generators of the cohomology ring of each component of the fixed point set with $\mathbb{Z}_p$-coefficient  is  at most the number of generators of $H^*(X).$ The fixed point sets of involutions on the product of projective spaces and Dold manifolds have been determined in \cite{chang, peltier}. \\
	\indent On the other hand, it is well known  that the orbit spaces of free actions of  $\mathbb{Z}_2,~ \mathbb{S}^1$ and $\mathbb{S}^3$ on  $\mathbb{S}^n,$ $\mathbb{S}^{2n+1}$ and $\mathbb{S}^{4n+3}$ are projective spaces $\mathbb{P}^n(q),$ where $q=1,2$ and $4,$ respectively.  Recently, the orbit spaces of free involutions on real Milnor manifolds, Dold manifolds, and the product of two projective spaces have been discussed in \cite{dey,Mattos,msingh}. The possibilities of connected fixed point sets of involutions and the orbit spaces of free involutions on the product of projective spaces and sphere $\mathbb{F}P^n \times \mathbb{S}^m,\mathbb{F}=\mathbb{R}$ or $\mathbb{C},$ have been determined in \cite{dimpi}.
In continuation, in this paper, we have determined the possibilities of the connected fixed point sets of involutions on $X\sim_2 \mathbb{H}P^n \times \mathbb{S}^m$ and discussed  the  orbit spaces of free involutions on $X.$ 
	
	\section{Preliminaries}  
	
	\noindent In this section, we recall  some known facts that will be used in this paper.  Let $G=\mathbb{Z}_p$ act on a finite CW-complex  $X.$ Let  $ G\hookrightarrow E_G  \rightarrow B_G$ be universal $G$-bundle, where  $E_G$ is contractible space and $B_G$ is finite CW-complex. Then the projection map $X\times E_G \rightarrow E_G$ is $G$-equivalent map and gives a  fibration $ X \stackrel{i} \hookrightarrow X_G \stackrel{\pi} \rightarrow B_G$ (called  Borel fibration), where $X_G = (X\times E_G)/G$ is  Borel space obtained by diagonal action of $G$ on space $X\times E_G.$ Suppose $F\neq \emptyset,$ and  let $x \in F$ and $\eta_x : B_G \hookrightarrow X_G$ be a cross section of projection map $\pi: X_G \rightarrow B_G ,$ where $B_G \approx (\{x\} \times E_G)/G ,$ then $H^*(X_G)\cong
	ker~\eta^*_x \oplus im~\pi^*.$  The induced homomorphism $\eta_x^*$ depends on the component $F_0$ of fixed point set $F$ in which $x$ lies. If $\alpha \in H^n(X_G)$ such that $\alpha \in Ker~ \eta_x^*$ then the image of $\alpha$ under the  restriction of $j: (F_G,x_G) \hookrightarrow (X_G,x_G)$ on $(F_0)_G$ does not  involved the elements of $H^0(F_0,x_G)$ \cite{G.Bredon}.\\
	Recall that a space $X$ is said to be  totally nonhomologous to zero (TNHZ)
	in $X_G$ if the inclusion map $i: X \hookrightarrow X_G$ induces
	a surjection in the cohomology $ i^*: H^*(X_ G) \rightarrow H^*(X). $ We
	have used the following Propositions:
	
	%	\begin{proposition}(\cite{bredon})
	%	For a finitistic $G$-space $X,$ where $G$ is cyclic group of prime order $p.$  We have   $ \sum_{i \geq j}$ rk $H^{i}(F;\mathbb{Z}_p) \leq \sum_{i \geq j} $ rk $H^{i}(X;\mathbb{Z}_p),$ for each $j.$
	%\end{proposition}
	\begin{proposition}(\cite{bredon})
		Let $G=\mathbb{Z}_2$  act on  a finite CW-complex $X$ and $\sum$ rk $H^i(X,\mathbb{Z}_2)< \infty.$ Then, the following statements are equivalent: \\
		(a) $X$ is TNHZ (mod 2) in $X_G.$   \\
		(b) $\sum$ rk $H^i(F,\mathbb{Z}_2) =\sum$ rk $H^i(X,\mathbb{Z}_2).$ \\ 
		(c) $G$ acts trivially on $H^*(X; \mathbb{Z}_2)$ and spectral sequence $E^{r,q}_2$ of $X_G \rightarrow B_G$ degenerates. 
	\end{proposition} 
	\begin{proposition}(\cite{Bredon})\label{2.2}
		Let $X$ be TNHZ in $X_G$ and $\{\gamma_j\}$ be a set of homogeneous elements in $H^*(X_G;\mathbb{Z}_p)$ such that $\{i^*(\gamma_j)\}$ forms $\mathbb{Z}_p$-basis of $H^*(X; \mathbb{Z}_p).$ Then, $H^*(X_G;\mathbb{Z}_p)$ is the free $H^*(B_G)$-module generated by $\{\gamma_j\}.$
	\end{proposition}
	\begin{proposition}\label{thm 2}(\cite{bredon})
		Let $G = \mathbb{Z}_2$ act on the finite CW-complex
		$X$ and  $A \subset X$ be closed and invariant subspace. Suppose that $H^i ( X , A ; \mathbb{Z}_2)
		= 0$ for $i > n.$ Then, the homomorphism       
		$$j ^*: H^ k ( X_G ,A_G; \mathbb{Z}_2) \rightarrow H^k(F_G, F_G \cap A_G; \mathbb{Z}_2)$$
		is an isomorphism for $k > n$. If $( X , A )$ is TNHZ (mod $2$) in $( X_ G , A_G,)$ then $j^*$ is a monomorphism for all $k.$ 
	\end{proposition}
	%	\begin{proposition} \label{thm 1}(\cite{bredon})
	%		Let $(X,A)$ be a finitistic poincar\'{e} duality pair over $\mathbb{Z}_p$ of formal dimension $n$, where $p$ is prime. Let $G=\mathbb{Z}_p$ acts on $X$ with $A$ invariant and such that $X$ is totally nonhomologous to zero in $X_G$ over $\mathbb{Z}_p.$ Then, for each component $F_0$ of $F=X^G,$ is a poincar\'{e} duality pair over $\mathbb{Z}_p$ of formal dimension $r \leq n.$ If $p \neq 2$ then $n-r$ is even. If $r=n$ then $F=F_0$ is connected  and  the restrictions $H^*(X;\mathbb{Z}_p)\rightarrow H^*(X;\mathbb{Z}_p)$ and $H^*(X,A;\mathbb{Z}_p\mathbb{Z}_2)\rightarrow H^*(X,F \cap A;\mathbb{Z}_p)$ are isomorphisms.
	%	\end{proposition}
	%	Let $ j_1: F \hookrightarrow X$ and $i_1: F \hookrightarrow F_G$ be inclusion maps. By the inclusion maps $i_1,i,j_1$ and $j,$ we get following commutative diagram (1):

	%		$$
	%	\begin{CD}
	%		H^*(X_G,x_G) @ > i^*>> H^*(X,x) \\
	%		@VV j^* V        @VV j_1^* V    \\
	%		H^*(F_G,x_G) @>>i_1^*> H^*(F,x) .
	%	\end{CD}
	%$$
	\begin{proposition}\label{prop 2.4}(\cite{bredon})
		Let $G=\mathbb{Z}_2$ act on a finite CW-complex $X$ and $X$ is TNHZ in $X_G.$ Then, $a|F$ is nontrivial element of the fixed point set $F,$ for any class $a\in H^n(X;\mathbb{Z}_2)$ such that $a^2 \neq 0.$
	\end{proposition}	
	
	\noindent We know that 
	$H^*(\mathbb{H}P^n \times \mathbb{S}^m; \mathbb{Z}_2)=\mathbb{Z}_2[a,b]/<{a^{n+1},b^2}>,$ where deg $a=4$ and  deg $b=m.$ 
	\noindent Throughout the paper,  
	$H^*(X)$ will denote the \v{C}ech cohomology of a space $X,$ and   $X\sim_2 Y,$  means $H^*(X;\mathbb{Z}_2 )\cong H^*(Y;\mathbb{Z}_2).$
	\section{Main Theorems}                               
	\noindent	Let $G=\mathbb{Z}_2$ act on  a finite CW-complex $X \sim_{2} \mathbb{H}P^n \times\mathbb{S}^m,$ where $ n, m \geq 1.$ In this section, we determined the possibilities of the connected fixed point sets  of involutions on $X,$ and  orbit spaces of free involutions on $X.$\\ %the fixed point set, for the case $n=1,$ that is,  $X \sim_2 \mathbb{S}^1 \times\mathbb{S}^m$ or $\mathbb{S}^2 \times\mathbb{S}^m,$ has been discussed by  Bredon \cite{Bredon}.
	First, we have determined the  fixed point sets of involutions on $X.$ 	
	\begin{theorem}\label{thm 3.3}
		Let $G=\mathbb{Z}_2$ act on a finite CW-complex $X \sim_{2} \mathbb{H}P^n \times\mathbb{S}^m,$ $ n,m\geq 1.$ If $X$ is TNHZ in $X_{G}$ and the fixed point set $F$ is nonempty and connected,   then $F$ must be one of the following:  
		\begin{enumerate}
			\item\label{1} $F\sim_2 \mathbb{S}^3\times \mathbb{S}^q$ or  $F \sim_{2} \mathbb{F}P^n \times\mathbb{S}^q,$ where $\mathbb{F}=\mathbb{R},$ $\mathbb{C}$ or  $\mathbb{H},$ $1 \leq q \leq m.$
			
			\item \label{2}  $F\sim_2 \mathbb{F}P^{n+1}\# \mathbb{F}P^{n+1},$ where  $\mathbb{F}=\mathbb{R},\mathbb{C}$ or $\mathbb{H}.$
			
			\item \label{4} $ H^*(F) $ is generated by $c$ and $d,$ with  $c^{n+1}=d^2+c^{s}=d^{2l+2}=0,$ where deg $c=2,$ deg $d=q$  and $l=[\frac{n}{s}],$ $s=\frac{q}{2}$ if $q$ is even and $s=q$ if $q$ is odd. Moreover, for $q=1,$ $F \sim_{2} \mathbb{R}P^{2n+1}$ and for $q=2,$ $F \sim_{2} \mathbb{C}P^{2n+1}.$

			\item \label{5}$ H^*(F) $ is generated by $c$ and $d,$ with  $c^{\frac{r}{s}+1}=d^{\frac{r}{q}+1}=c^{\frac{r}{s}}+d^{\frac{r}{q}}=cd=0,$ where deg $c=s,s=1,2 ,$ deg $d=q,q=1,2,4$ or $8,$ $r=sq(2n+2)/(q+s)$ and $n=\frac{(q+s)k}{2}-1$ for  some $k\in \mathbb{N}.$ 
			
			\item \label{6}$ H^*(F) $ is generated by $c$ and $d,$ with   $c^{\frac{r}{s}+1}=c^{\frac{qj}{s}}+d^{j}=c^{\frac{r-q}{s}+1}d=0,$ where deg $c=s,s=1,2,$ deg $d=q< n,$   $r=\frac{s(2n+2)}{j}+qj-(q+1),n+1= jk$ for some $k\in \mathbb{N},$ or $j=2k,k=1$ or $2,$  and $n>\frac{(q+1)j}{2s}-1.$

		\end{enumerate}
	\end{theorem}
	\begin{proof}
		Let $x \in  F$ and   $\{a,\cdots, a^n,b, ab,\cdots a^nb\}$  be a  generating set of $H^*(X,x),$ where deg $a=4$ and deg $b=m.$  Since $X$ is TNHZ in $X_G,$ we get  rk $H^*(F)=2n+2,$ $\pi_1(B_G)$ acts trivially on $H^*(X_G,x_G)$ and the $E_2$-term $E_2^{p,q}=H^p(B_G) \otimes H^q(X)$ of Leray-Serre spectral sequence of the Borel fibration $ X \stackrel{i} \hookrightarrow X_G \stackrel{\pi} \rightarrow B_G$  is $E_{\infty}^{p,q}.$ So, the elements $\{1 \otimes a,1 \otimes a^2,\cdots, 1 \otimes a^n,1 \otimes b ,1 \otimes ab,\cdots 1 \otimes a^nb\}$ are permanent cocycles. Assume that $\alpha \in H^4(X_G,x_G)$ represents  generator $a \in H^4(X,x)$ and $\beta \in H^m(X_G,x_G)$ represents  generator $b \in H^m(X,x)$ such that $\eta_x^*(\alpha)=\eta_x^*(\beta)=0,$ where $\eta_x: \frac{\{x\} \times E_G}{G} \hookrightarrow \frac{X \times E_G}{G}$ is the inclusion map.  By Proposition \ref{2.2}, $\{\alpha, \alpha^2, \alpha^3, \cdots \alpha^n, \beta ,\alpha \beta , \alpha ^2 \beta, \cdots \alpha^n \beta\}$ is a generating set of $H^*(X_G,x_G)$ over $H^*(B_G)$-module. As
	   	$ H^m(F_G,x_G)= \bigoplus_{i=0}^{m} H^{m-i}(B_G) \otimes H^i(F,x)$ and $\eta_x(\beta)=0.$  We may assume that
		\begin{equation*}
			j^*(\beta)=1 \otimes d_m +  t \otimes d_{m-1} + \cdots + t^k \otimes d_{m-k} \cdots + t^{m-2}\otimes d_{2}+ t^{m-1}\otimes d_1,
		\end{equation*} where $d_i \in H^i(F,x),$ and  $j^*(\alpha)= B_1 t^3 \otimes c_1 + B_2 t^2 \otimes c_2+B_3 t \otimes c_3 + B_4 1 \otimes c_4,$ where $c_i\in H^i(F,x) $ and $B_i \in  \mathbb{Z}_2,1\leq i \leq 4.$ We know that $i_1^*j^*=j_1^*i^*$ where $i_1: F \hookrightarrow F_G$ and $ j_1: F \hookrightarrow X$  are the inclusion maps. So,  we get  $c_4=a|F.$ 
		If $c_4\neq 0,$   then $B_4=1.$ Clearly, $c_4^{n+1}=0.$  Thus, $j^*(\alpha)=1 \otimes c_4 + \sum_{i=1}^{3}B_it^{4-i}\otimes c_i,$ where $B_i \in \mathbb{Z}_2, 1 \leq i \leq 3.$ So, we  consider  eight cases according as $B_1,B_2$ and $B_3$ are zero or nonzero.  \\
		\textbf{Case (1):} If $B_1=B_2=B_3=0,$ then $j^*(\alpha)=1 \otimes c_4.$\\
		In this case, $c_4^i \neq 0 $ for $1\leq i \leq n.$ As $j^*$ is injective, $d_j \neq c_4^j,$ for some $j,$ $1 \leq j \leq n,$ where deg $d_j=j.$ Suppose $d_q=d_{m-k}=d$ is the least degree element such that $d_q \neq c_4^j.$ %Clearly, if $H^1(F,x)\neq 0 $ then $d=d_1$ otherwise $d_1=0.$
		As $j^*$ is onto on high degrees, for sufficiently large value of $r,$ we can write
		\begin{align*}
			t^{k+r} \otimes d=j^*(A_1 t^{r+m-4}\alpha  +\cdots +A_n t^{r+m-4n}\alpha^n+ A_m t^{r}\beta + \cdots+A_{m+n} t^{r-4n}\alpha^{n} \beta),
		\end{align*}
		where $A_i's$ are in $\mathbb{Z}_2.$ After comparing the coefficient of $t^{k+r}\otimes d,$ we get $A_m=1.$ So, we have\\ 	
	$t^{r} \otimes d_m+\cdots +t^{r+k-1} \otimes d_{m-(k-1)}+t^{r+k+1} \otimes d_{m-(k+1)}+ \cdots + t^{r+m-1} \otimes d_1	= -j^*(A_1 t^{r+m-4}\alpha + \cdots + A_n t^{r+m-4n}\alpha^n + A_{m+1} t^{r-4}\alpha \beta  + \cdots+A_{m+n} t^{r-4n}\alpha^{n} \beta).$\\
	From the above equation, we get that if $q\equiv 1,2,3~ (\mbox{mod}~ 4),$ then  $d_{4i}=c_4^i, d_{4i+q}=c_4^id ,1 \leq i \leq n,$ and if $q\equiv 0  ~(\mbox{mod}~ 4),$ then $ d_{4k}= c_4^{k}+c_4^{k-1}d,$ where $1 \leq i \leq n,$ and zero otherwise. Thus, we get $j^*(\alpha^n\beta )=t^k\otimes c_4^nd.$ As $\alpha^n\beta \neq 0, $ we get  $c_4^nd \neq 0,$ and hence $c_4^id \neq 0$ for $ 1 \leq i \leq n.$ Clearly, if $d^2=0,$ then $F\sim_2 \mathbb{H}P^n \times \mathbb{S}^q, 1 \leq q \leq m.$ If $d^2 \neq 0,$ then either   $q\equiv 0$(mod $4)$ or $q\equiv 2 $(mod $4$). First, suppose that $q\equiv 2 $(mod $4$). So, we have  $d^2=c_4^{\frac{q}{2}}.$ Consequently, $d^{2l+1}=c_4^{\frac{lq}{2}}d,$ where $l=[\frac{2n}{q}]$ and $d^{2l+2}=0.$ Thus, we get $c_4^{n+1}=d^{2l+2}=d^2+c_4^{\frac{q}{2}}=0.$ In particular, for  $q=2,$ we get $F \sim_2 \mathbb{C}P^{2n+1}.$ This realizes possibility (\ref{4}). Next, suppose that $q\equiv 0 $(mod $4$) then we have either  $d^2=c_4^{\frac{q}{2}}$ or $d^2=c_4^{\frac{q}{4}}d.$ If $d^2=c_4^{\frac{q}{2}},$ then by suitable change of basis, we get $F\sim_{2} \mathbb{H}P^n \times \mathbb{S}^q, 1 \leq q \leq m.$ If $d^2=c_4^{\frac{q}{4}}d,$ then $q$ must be $4.$ By the change of basis $d'=d+c_4,$ we get $d'^{n+2}=d^{n+2}=d'^{n+1}+d^{n+1}=dd'=0.$ Thus, $F\sim_2 \mathbb{H}P^{n+1}\# \mathbb{H}P^{n+1}.$ This realizes possibility (2) for $\mathbb{F}=\mathbb{H}.$\\
		\textbf{Case (2): }  If $B_1=B_2=0$ and $B_3=1,$ then $j^*(\alpha)= 1 \otimes c_4 + t \otimes c_3.$\\		
		Assume that  $H^1(F)\neq 0.$ Further, we consider cases according as $c_3=c_1^3$ or   $c_3\neq c_1^3.$\\
		First, consider $c_3=c_1^3.$ Suppose $H^*(F)$ has one generator. Then, $c_4=c_1^4$ and $j^*(\alpha^n)=\sum_{r=0}^{n} t^r \otimes c_1^{4n-r}.$ By the injectivity of homomorphism  $j^*,$ we get $c_1^{3n}\neq 0.$ Clearly, rk $H^*(F) >2n+2,$ a contradiction. Suppose $H^*(F)$ has two generators. Then, either $c_4=c_1^4$ or $c_4\neq c_1^4.$ \\
		Let  $c_4=c_1^4.$ Then, we also have  rk $H^*(F)>2n+2,$ a contradiction.\\
		Let $c_4\neq c_1^4.$ Further, if $c_1^4=0,$ then 
		\[
		j^*(\alpha^n)=
		\begin{cases}
			1 \otimes c_4^n & \mbox{if $n$ is even} \\
			1\otimes c_4^n + t \otimes c_4^{n-1}c_1^3 & \text{if $n$ is odd}. \\
		\end{cases}
		\]
		As $j^*(\alpha^n) \neq 0,$ $c_4^n\neq 0,$ for $n$ even. If  $n$ is  odd and $c_4^n= 0,$ then rk$H^*(F)=4n> 2n+2, n>1$ a contradiction. Clearly, this case is not possible for $n=1.$ Thus, we have  $c_4^n \neq 0.$ As $F$ is poincar\'{e} duality space, again we get  rk $H^*(F)> 2n+2,$ a contradiction.\\
		Now, if $c_4\neq c_1^4~\&~c_1^4\neq 0,$ then we must have $$j^*(\alpha^n \beta )= \sum_{r=0}^{n}\sum_{i=0}^{m-1}{{n}\choose{r}} t^{r+i} \otimes (\oplus_{j+4l=m-i} c_4^{n-r+l}c_1^{j+3r}).$$ If the  cup product  $c_1c_4=0,$ then $c_1^{3n+1}\neq 0.$ This gives rk $H^*(F)>2n+2,$ a contradiction. If the cup product $c_1c_4\neq 0,$ then the rank of $H^*(F)$ further increases, which is not possible.\\
		Next, consider $c_3 \neq c_1^3.$ As $H^*(F)$ has at most two generators, we get $c_4=c_1^4.$ Thus $$j^*(\alpha^n \beta )= \sum_{r=0}^{n}\sum_{i=0}^{m-1}{{n}\choose{r}} t^{r+i} \otimes (\oplus_{j+3l=m-i} c_1^{4n-4r+j}c_3^{r+l}).$$ We get $c_1^{4n+1},c_1^{4n-3}c_3,\cdots ,c_1c_3^n$ are least degree elements when $r=0,1, \cdots , n,$ respectively.  In any case, it is easy to observed that  rk $H^*(F)>2n+2,$ a contradiction.
		
		Now, Suppose that $H^1(F) =0.$ Further, assume that $H^2(F)\neq 0.$ Then, we must have  $c_4 =c_2^2.$ Then, $$j^*(\alpha^n \beta )= \sum_{r=0}^{n}\sum_{i=0}^{m-1}{{n}\choose{r}} t^{r+i} \otimes (\oplus_{2l+3j=m-i} c_2^{2n-2r+l}c_3^{r+j}).$$ Note that $c_2^{2n+1},c_2^{2n-1}c_3.c_2^{2n-3}c_2^2. \cdots c_2^2c_3^{n-1},c_2c_3^n$ are the least degree elements when $r=0,1,2 \cdots n$ respectively. So, we always have rk $H^*(F)>2n+2,$ a contradiction.  \\
		Next, assume that $H^2(F)=0.$ Then, 	 $$j^*(\alpha^n \beta )= \sum_{r=0}^{n}\sum_{i=0}^{m-1}{{n}\choose{r}} t^{r+i} \otimes (\oplus_{3l+4j=m-i} c_4^{n-r+j}c_3^{r+l}).$$ The  least degree elements in the above expression $c_4^nc_3,c_4^{n-1}c_3^2, \cdots c_4c_3^n$ and $c_3^{n+1}.$ Let $c_3c_4\neq 0.$ If $c_3^2=0,$ then $c_4^nc_3\neq 0.$ Thus, $F\sim_2 \mathbb{H}P^n \times \mathbb{S}^3.$ If $c_4^2=0$ and $c_3^{n+1}=0 ,$ then $F\sim_2 X \times \mathbb{S}^4,$ where $X$ has truncated polynomial ring $\frac{\mathbb{Z}_2[x]}{<x^{n+1}>},$  deg $x=3.$ By the Theorem $4.5$ in \cite{steenrold}, this is not possible. If $c_4^2=0$ and $c_3^{n+1}\neq 0,$ then  rk $H^*(F)>2n+2,$ a contradiction.\\ Now, let $c_3c_4=0.$ Then, $c_3^{\frac{r}{3}}=c_4^{\frac{r}{4}}$ is generator of $H^r(F)$ and $c_3^{\frac{i}{3}} \neq c_4^{\frac{i}{4}}$ for $i<r.$ As, rk $H^*(F)=2n+2,$ we get $r=\frac{24n+24}{7}.$ So, $n = \frac{7k-2}{2},k$ even. Thus, $c_3^{4k+1}=c_4^{3k+1}=c_3c_4=c_3^{4k}+c_4^{3k}=0.$ Thus, $F\sim_2 Y \# Z,$ where $Y$ and $Z$ both are truncated polynomials with generator $c_3$ and $c_4,$ respectively. But by the Theorem $4.5$ in \cite{steenrold}, this is not possible.\\
		\textbf{Case (3):}  If $B_1=B_3=0$ and $B_2=1,$ then   $j^*(\alpha)=1 \otimes c_4 + t^2 \otimes c_2.$\\
		If $H^*(F)$ has one generator then $F\sim_{2} \mathbb{R}P^{2n+1}$ or $F \sim_2 \mathbb{C}P^{2n+1}$ according as $H^1(F)\neq 0$ or $H^1(F)=0.$ This realizes possibility (\ref{4}) for $q=1$ and $q=2,$ respectively. Now, assume that $H^*(F)$ has two generators.  We consider two subcases according as $c_4\neq c_2^2$ or  $c_4= c_2^2.$\\
		\textbf{Subcase(i):} Assume that $c_4 \neq c_2^2.$\\ First, assume that $H^1(F)=0.$ If $c_2^2=0$ then  	\[
		j^*(\alpha^n)=
		\begin{cases}
			1 \otimes c_4^n & \mbox{if $n$ is even} \\
			1 \otimes c_4^{n} +t^{2}\otimes c_4^{n-1}c_2 & \text{if $n$ is odd}. \\
		\end{cases}
		\] By the  injectivity of $j^*,$ we get $c_4^n\neq 0.$ If $c_4^{n+1}=0$ then $F\sim_2 \mathbb{H}P^n \times \mathbb{S}^2.$ If $c_4^{n+1}\neq 0$ then rk $H^*(F)>2n+2,$ a contradiction.\\
		If $c_2^2 \neq 0,$ then we get $j^*(\alpha^n\beta)= \sum_{r=0}^{n}\sum_{i=0}^{m-1}{{n}\choose{r}}t^{2r+i} \otimes (\oplus_{2l+4j=m-i}c_4^{n-r+j}c_2^{r+l}).$ We get that the least degree elements are $c_4^nc_2,c_4^{n-1}c_2^2,c_4^{n-2c_2^3}\cdots c_4^2c_2^{n-1},c_4c_2^n$ and $c_2^{n+1}.$ Note that if $c_4^{n-k}c_2^{k+1}\neq 0, $ for any $0 \leq k \leq n-1,$ then rk $H^*(F)>2n+2,$ a contradiction. So, at least one of $c_2^{n+1}$ or $c_2^nc_4$ must be nonzero.\\
		Let $c_2c_4=0,$ then we must have   $c_2^{n+1}\neq 0.$ Thus, $c_2^{\frac{r}{2}}=c_4^{\frac{r}{4}}$ is the generator of $H^r(F).$ This implies that  rk $H^*(F)={\frac{r}{2}}+{\frac{r}{4}}=2n+2.$ Thus, $r=8k,$ where $n=3k-1,$  $k \in \mathbb{N}.$ Hence, $F\sim_2 \mathbb{C}P^{2k}\# \mathbb{H}P^{2k}, k \in \mathbb{N}.$ This realizes possibility (\ref{5}) for $s=2~\&~q=4.$\\
		If  $c_2c_4 \neq 0,$ then for $c_2^{n+1}=0~\&~$$c_4^2=0,$ clearly, $F\sim_2 \mathbb{C}P^n \times \mathbb{S}^4.$ For $c_4^2 \neq 0,$  we must have $c_4^2=c_2^3.$ By the change of basis $d'=c_2^2+c_4,$ we get the cohomology ring is given by $c_2^{n+1}=d'^2=0.$ This realizes possibility (1) for $\mathbb{F}=\mathbb{C} ~\&~ q=4.$\\
		If $c_2^{n+1}\neq 0$ and $c_2^nc_4 \neq 0,$ then rk $H^*(F)>2n+2,$ a contradiction.\\
		If $c_2^{n+1}\neq 0.$ Then, $c_2^{\frac{r}{2}}=c_2^{\frac{r-4j}{2}}c_4^j,$  $j>1$ forms  generator of $H^r(F).$ Which implies  that $c_2^{2j}=c_4^{j}$ is generator of $H^{4j}(F).$ We get rk $H^*(F)=j(\frac{r-4j}{2})+\frac{4j}{2}+j$ which must be  $2n+2$ so, $r= \frac{4n+4}{j}+4j-2.$ We must have either $n+1= jk,$ for some $k\in \mathbb{N}$ or $j=2k, k=1,2.$  Note that  $c_2^{\frac{r-4j}{2}+1}c_4=0,$ as if $c_2^{\frac{r-4j}{2}+1}c_4\neq 0.$ Then,  poinca\'{r}e dual of $c_2^{\frac{r-4j}{2}+1}c_4$ are namely, $c_2^q$ and $d^2,$ which is not possible. Hence, the cohomology ring is given by $c_2^{\frac{r}{2}}=c_2^{2j}+c_4^{j}= c_2^{\frac{r-4j}{2}+1}c_4=0.$ This realizes possibility (\ref{6}) for $s=2~\&~ q=4.$\\
		Now, suppose that $H^1(F) \neq 0,$ then we consider two possibility accordingly $c_2=c_1^2$ or $c_2 \neq c_1^2.$\\
		If $c_2=c_1^2,$ then we get $c_1^{2n}\neq 0.$ which leads contradiction.\\
		If $c_2 \neq c_1^2,$ then we must have $c_4=c_1^4.$ It is easy to observed that either $c_2^2$  is zero or non zero, we get rk $H^*(F)>2n+2,$ a contradiction.\\
		\textbf{Subcase(ii):} Assume that $c_4=c_2^2.$\\
		If $H^1(F)\neq 0,$ then for $c_2=c_1^2,$ we get $c_1^{2n+1}$ must be non zero. Thus, rk $H^*(F)>2n+2,$ a contradiction. Now, for $c_2\neq c_1^2,$ we have  $j^*(\alpha^n)= \sum_{r=0}^{n}{{n}\choose {r}}t^{2r}\otimes c_2^{2n-r},$ which implies that  $c_2^n$ must be nonzero.  If $c_1^2=0$ and  $c_2^{n+1}=0,$ then $F \sim_2 \mathbb{C}P^n \times \mathbb{S}^1.$ Which realizes possibility (1) for $\mathbb{F}=\mathbb{C}$ and $q=1.$ If $c_2^{n+1}\neq 0,$ then rank of $H^*(F)$ exceed $2n+2,$ a contradiction. Now, suppose that $c_1^2\neq 0$ then this case only possible when $n=2$ and cup product $c_1c_2=0,$ otherwise rk $H^*(F)>2n+2.$ For $n=2,$ we get $c_1^4=c_2^2$ is generator of  $H^4(F).$ Thus, $F\sim_2 \mathbb{R}P^4 \# \mathbb{C}P^2.$ This realizes possibility (\ref{5}) for $n=2$ and $q=2~\&~s=1.$\\
		If $H^*(F)=0,$ then  $j^*(\alpha^n )=\sum_{r=0}^{n}{{n}\choose {r}}t^{2r}\otimes c_2^{2n-r},$ which implies that $c_2^n\neq 0.$\\ For $c_2^{n+1}=0,$ we get $j^*(\alpha^n\beta )=c_2^nd\neq 0,$ where deg $d=q,$  $d\neq c_2^i, 1\leq i \leq n.$ If $d^2=0,$ then $F\sim_2 \mathbb{C}P^n \times \mathbb{S}^q, 1 \leq q \leq m.$ If $d^2 \neq 0, $ then either $d^2=c_2^q$ or $d^2=c_2^{\frac{q}{2}}d.$\\
		Suppose that $d^2=c_2^q,$ then if $q \equiv 0$ (mod $2),$ then by the change of basis $d'=d+c_2^{\frac{q}{2}},$ we get $d'^2=0$ and $c_2^id'\neq 0$
		for $1\leq i\leq n.$ Thus, $F \sim_2 \mathbb{C}P^n \times \mathbb{S}^q, 1 \leq q \leq m.$ This realizes possibility (1) for $\mathbb{F}=\mathbb{C}.$ If $q \not\equiv 0$ (mod $2),$ then $d^{2l+1}=c_2^{lq}d,$ where $l=[\frac{n}{q}]$ and $d^{2n+2}=0.$ This realizes possibility (\ref{4}) for $s=q~\&~q $ odd. Moreover, if $q=1,$ then clearly $F\sim_2 \mathbb{R}P^{2n+1}.$\\
		If $d^2=c_2^{q/2}d,$ then $q$ must be $2.$ As if $2<q(\mbox{even}) \leq n,$ then $F$ does not satisfy poinca\'{r}e duality.  By the change of basis $d'=d+c_2,$ we get $d'^{n+2}=d^{n+2}=d'^{n+1}+d^{n+1}=dd'=0 .$ Thus, $F\sim_2 \mathbb{C}P^{n+1}\# \mathbb{C}P^{n+1}.$   This realizes possibility (2) for $\mathbb{F}=\mathbb{C}.$ \\
		For $c_2^{n+1}\neq 0,$ we get	$j^*(\alpha^n \beta )=\sum_{r=0}^{n}\sum_{i=1}^{m-1}{{n}\choose{r}}t^{2r+i}\otimes (\oplus_{2l+qj=m-i} c_2^{2n-r+l}d^j )_{j=\{0,1\}}$  \begin{equation*}
			=\sum_{r=0}^{n}\sum_{i=1}^{m-1}{{n}\choose{r}}t^{2r+i}\otimes c_2^{2n-r+\frac{m-i}{2}}+\sum_{r=0}^{n}\sum_{i=q}^{m-1}A_{l,j}{{n}\choose{r}}t^{2r+i}\otimes (\oplus_{2l+q=m-i}c_2^{2n-r+l}d).
		\end{equation*}
		From above expression we get $c_2^{n+1} $ amd $c_2^nd$ are least degree elements. Clearly, if $c_2^nd \neq 0$ then rk $H^*(F)>2n+2,$ a contradiction.
		Now, suppose that $c_2^nd=0~\&~c_2^{n+1}\neq 0.$ \\
		If  $c_2d=0,$ then we must have $d^2\neq 0, $ otherwise, we cannot have poincar\'{e} dual of $d.$ It is easy to observe that   $c_2^{\frac{r}{2}}=d^{\frac{r}{q}}$ is the generator of $H^r(F),$ where $r$ is the formal dimension of $H^*(F).$ As rk $H^*(F)=2n+2,$ we get $r= \frac{4q(n+1)}{q+2}.$ So, $(q+2)|(4n+4),$ and hence, $n=(q+2)k-1$ for $q\equiv 0,1,$ or $3$ (mod $4$)  $\&$ $n=(\frac{q+2}{2})k-1$ for $q\equiv 2$ (mod $4).$
		Thus, $c_2^{\frac{r}{s}+1}=d^{\frac{r}{q}+1}=c_2^{\frac{r}{s}}+d^{\frac{r}{q}}=c_2d=0.$ This realizes possibility (\ref{5}) for $s=2.$	
		
		If  $c_2d\neq 0$ and $c_2^nd=0.$ Let $r$ be the formal dimension of $F.$ In this case,  
		we show that the generators of $H^r(F)$ would be $c_2^{\frac{r}{2}}$ which is equal to   $c_2^{\frac{r-qj}{2}}d^j,$ where $j>1.$  Assume that $c_2^{\frac{r}{2}}= 0.$  Now, if $c_2^{\frac{r-q}{2}}d$ is generator of $H^r(F),$ then $c_2^{\frac{i}{2}} \neq c_2^{\frac{i-q}{2}}d, q <i <r,$ otherwise, $c_2^{\frac{r}{2}}=c_2^{\frac{r-q}{2}}d$, a contradiction. Thus, rk $H^*(F)\geq 2n+4>2n+2,$ which contradicts our hypothesis. Similarly, $c_2^{\frac{r-qj}{2}}d^j,$ where $j>1,$ cannot be generator of $H^r(F).$ Thus, $c_2^{\frac{r}{2}}\neq 0.$  Further, if $c_2^{\frac{r}{2}}=c_2^{\frac{r-q}{2}}d$ is a generator of $H^r(F),$ then $c_2^{\frac{r-q}{2}}$  would have two poincar\'{e} duals namely, $c_2^{\frac{q}{2}}$ and $d,$ again a contradiction. Hence, $c_2^{\frac{r}{2}}=c_2^{\frac{r-qj}{2}}d^j$ is generator of $H^r(F),$ where $j>1.$
		
		As $c_2^nd=0,$ we must have  $q\leq n$ and $c_2^{\frac{r}{2}}=c_2^{\frac{r-qj}{2}}d^j,j>1,$ generates  $H^r(F),$ where $r$ is  even.  Thus,    $c_2^\frac{qi}{2}{}\neq d^i$ for $1\leq i \leq j-1$  and $c_2^{\frac{qj}{2}}= d^j,$ where $qj<r.$ We get rk $H^*(F)=j(\frac{r-qj}{2})+\frac{qj}{2}+j$ which must be  $2n+2$ so, $r= \frac{4n+4}{j}+qj-(q+2).$ We must have either $n+1= jk,$ for some $k\in \mathbb{N}$ or $j=2k,k=1,2.$   Hence, the cohomology ring $H^*(F)$ is generated by $c_2$ and $d,$ $c_2^{\frac{r}{2}+1}=c_2^{\frac{qj}{2}}+d^{j}=c_2^{\frac{r-qj}{2}+1}d=0.$ As $qj<r,$ we get   $\frac{(q+1)j}{4}-1<n.$ %Moreover, if $q$ is odd, then $j=4k,$ and $n\equiv lk-1,$ where $l$ is odd, $k\in \mathbb{N}.$ 
		This realizes possibility (5) for $s=2.$ \\
		\textbf{Case (4): }  If $B_2=B_3=0$ and $B_1=1,$ then  $j^*(\alpha)=1 \otimes c_4+t^3\otimes c_1.$\\
		In this case, if $H^*(F)$ has one generator then $F\sim_2 \mathbb{R}P^{2n+1}.$  Suppose that $H^*(F)$ has two generators. Now, we consider two subcases according as $c_4=c_1^4$ or $c_4\neq c_1^4.$\\
		\textbf{Subcase(i):} Assume that $c_4=c_1^4.$\\
		We have $j^*(\alpha^n)= \sum_{r=0}^{n}{{n}\choose{r}}t^{3r}\otimes c_1^{4n-3r}.$ Which implies that $c_1^n$ must be non zero. Let  $d\neq c^i,1\leq i \leq n$ be the generator of $H^*(F)$ having degree $q.$  We get \begin{equation*}
			j^*(\alpha^n \beta )=\sum_{r=0}^{n}\sum_{i=1}^{m-1}{{n}\choose{r}}t^{3r+i}\otimes c_1^{4n-3r+m-i}+\sum_{r=0}^{n}\sum_{i=q}^{m-1}A_{l,j}{{n}\choose{r}}t^{3r+i}\otimes c_1^{4n-3r+(m-i-q)}d.
		\end{equation*} 
		After expanding the above expression we get $c_1^{n+1}$ and $c_1^nd$ are the least possible  degree elements. If $c_1^{n+1}=0,$ then $c_1^nd$ must be non zero. Thus, for $d^2=0,$ we get $F \sim_2 \mathbb{R}P^n \times \mathbb{S}^q, 1 \leq q\leq m.$ This realizes possibility (1) for $\mathbb{F}=\mathbb{R}.$ And for $d^2\neq 0 ,$ we have two possibility either $d^2=c_1^{2q}$ or $d^2=c_1^qd.$\\
		If $d^2=c_1^{2q},$ then by the change of basis $d'=d+c_1^q,$ we get realizes possibility (1) for $\mathbb{F}=\mathbb{R}.$\\
		If $d^2=c_1^qd,$ then for $1<q\leq n,$  $F$ does not satisfy poincar\'{e} duality. So, we must have $q=1.$ Again, by the change of basis $d'=c+d,$  we get $d'^{n+2}=d^{n+2}=d'^{n+1}+ d'^{n+1} = dd'=0.$ Thus, $F\sim_2 \mathbb{R}P^{n+1}\#\mathbb{R}P^{n+1}.$ This realizes possibility (2) for $\mathbb{F}=\mathbb{R}$. \\
		Now, Assume that $c_1^{n+1}\neq 0$ then $c_1^nd$ either zero or non zero.
		Obviously, $c_1^nd\neq 0$ is not possible. Suppose that $c_1^nd =0.$ \\ If $c_1d=0,$ then  we get $c_1^{r+1}=d^{\frac{r}{q}+1}=c_1d=c_1^r+d^{\frac{r}{q}}=0,$ where $r=\frac{q(2n+2)}{q+1}.$ This realizes possibility (\ref{5}) for $s=1.$ \\If $c_1d\neq 0,$ then  we get $c_1^r=c_1^{r-qj}d^j,j>1$ which generates  $H^r(F).$  Thus,    $c_1^{qi}\neq d^i$ for $1\leq i \leq j-1$  and $c_1^{qj}= d^j$ where $qj<r.$ We get rk $H^*(F)=j(r-qj)+qj+j$ which must be  $2n+2$ so, $r= \frac{2n+2}{j}+qj-(q+1).$ We must have either $n+1= jk,$ for some $k\in \mathbb{N}$ or $j=2.$   Hence, the cohomology ring $H^*(F)$ is generated by $c_1, d$ with $c_1^{r+1}=c_1^{qj}+d^{j}=c_1^{r-qj+1}d=0,$ with $\frac{(q+1)j}{2}-1<n.$ This realizes possibility (\ref{6}) for $s=1.$\\
		\textbf{Subcase(ii):} Assume that $c_4 \neq c_1^4.$\\
		First, we consider when $c_1^4=0.$\\
		If $c_1^2=0,$ then \[ j^*(\alpha^n)=
		\begin{cases}
			1 \otimes c_4^n & \mbox{if $n$ is even} \\
			1 \otimes c_4^{n} +t^{3}\otimes c_4^{n-1}c_1 & \text{if $n$ is odd}. \\
		\end{cases}
		\] 
		Clearly, $c_4^n$ must be nonzero. Thus, $F\sim_2 \mathbb{H}P^n \times \mathbb{S}^1.$ This realizes possibility (1) for $\mathbb{F}=\mathbb{H}$ and $q=1.$\\
		If $c_1^2\neq 0$ and $c_1^3=0,$ then $j^*(\alpha^n) = \sum_{r=0}^{2}{{n}\choose {r}}(1\otimes c_4)^{n-r}(t^3\otimes c_1)^r$ and  $$j^*(\alpha^n\beta)= \sum_{i=0}^{m-1}\sum_{r=0}^{2}{{n}\choose {r}}t^{i+3r} \otimes (\oplus_{4l+j=m-i}c_4^{n-r+l}c_1^{r+j}).$$
		So, we get $c_4^nc_1$ and $c_4^{n-1}c_1^2$ are the least degree elements. Since $F$ is poinca\'{r}e duality space and rk $H^*(F)=2n+2,$ so $c_4^{n-1}c_1^2$ is generator of formal dimension. So, this possible only when $n=2$ and $c_4^2=0.$ Thus, $F\sim_2 \mathbb{R}P^2\times \mathbb{S}^4.$\\
		If $c_1^3\neq 0$ and $c_1^4 =0,$ then $j^*(\alpha^n\beta)= \sum_{i=0}^{m-1}\sum_{r=0}^{3}{{n}\choose {r}}t^{i+3r} \otimes (\oplus_{4l+j=m-i}c_4^{n-r+l}c_1^{r+j}).$ Then, $c_4^{n-1}c_1^2$ and $c_4^{n-2}c_1^3$ are the possible  least degree elements. Clearly, $c_4^{n-2}c_1^3$ is possible generator of formal dimension only when $n=3$ and $c_4^2=0.$ Thus we get $F\sim_2 \mathbb{R}P^3\times \mathbb{S}^4.$\\
		Now, suppose that $c_1^4\neq 0.$ \\
		We have  $j^*(\alpha^n\beta)= \sum_{i=0}^{m-1}\sum_{r=0}^{n}{{n}\choose {r}}t^{i+3r} \otimes (\oplus_{4l+j=m-i}c_4^{n-r+l}c_1^{r+j}).$ It is easy to observed that $c_1^{n+1}$ and $c_1^nc_4$ are the least possible degree elements.\\
		If $c_1^{n+1}=0,$ then $c_1^nc_4$ must be non zero. Clearly, when $c_4^2=0,$ $F\sim_2 \mathbb{R}P^n\times \mathbb{S}^4$  and when $c_4^2\neq 0,$ then we must have  $c_4^2=c_1^8.$ After change of basis $d'=c_1^4+c_4$ we get $F\sim_2 \mathbb{R}P^n\times \mathbb{S}^4.$ This realizes possibility (1) for $\mathbb{F}=\mathbb{R}$ and $q=4.$ \\
		Now suppose that $c_1^{n+1}\neq 0,$ then $c_1^nc_4$ either zero or non zero.\\
		Obviously, $c_1^nc_4\neq 0$ is not possible. And if $c_1^nc_4 =0,$ then for cup product zero, we get $c_1^{r+1}=c_4^{\frac{r}{4}+1}=c_1c_4=c_1^r+c_4^{\frac{r}{4}}=0$ where $r=\frac{4(2n+2)}{5}.$ This realizes possibility (\ref{5}) for $s=1~\&~q=5.$ For cup product non zero,  we get $c_1^r=c_1^{r-4j}c_4^j,j>1$ which generates  $H^r(F).$  Thus,    $c_1^{4i}\neq c_4^i$ for $1\leq i \leq j-1$  and $c_1^{4j}= c_4^j,$ where $4j<r.$ We get rk $H^*(F)=j(r-4j)+4j+j$ which must be  $2n+2$ so, $r= \frac{2n+2}{j}+4j-5,$ where  either $n+1= jk,$ for some $k\in \mathbb{N}$ or $j=2.$     Hence, the cohomology ring $H^*(F)$ is generated by $c_1$ and $ c_4$ with $c_1^{r+1}=c_1^{4j}+c_4^{j}=c_1^{r-4j+1}c_4=0,$ with $\frac{5j}{2}-1<n.$ This realizes possibility (\ref{6}) for $s=1$ and $q=4.$\\
		\textbf{Case(5):}  If $B_1=0$ and $B_2=B_3=1,$ then $j^*(\alpha)=1 \otimes c_4+t\otimes c_3+t^2\otimes c_2.$\\
		In this case, if $H^*(F)$ has one generator then $F\sim_2 \mathbb{R}P^{2n+1}.$ Suppose that $H^*(F)$ has two generators. Now, we consider two subcases: (i) $c_4=c_2^2$ (ii) $c_4\neq c_2^2.$\\
		\textbf{Subcase(i):} $c_4=c_2^2$\\
		If $H^1(F)=0,$ then  $j^*(\alpha^n) = \sum_{r=0}^{n}\sum_{k=0}^{n-r}{{n}\choose {r}}{{n-r}\choose {k}}t^{2k+r}\otimes c_2^{2n-2r-k}c_3^{r}$ and  $$j^*(\alpha^n\beta)= \sum_{i=0}^{m-1}\sum_{r=0}^{n}\sum_{k=0}^{n-r}{{n}\choose {r}}{{n-r}\choose {k}}t^{2k+r+i} \otimes (\oplus_{2l+3j=m-i}c_2^{2n-2r-k+l}c_3^{r+j}).$$ 
		If $c_2c_3=0,$ then we get $c_2^{n+1}$ and $c_3^{n+1}$ are the least degree element. We must have $c_2^{\frac{r}{2}}=c_3^{\frac{r}{3}}$ is generator of $H^r(F).$ Thus, rk $H^*(F)=\frac{r}{2}+\frac{r}{3}=2n+2\implies r=\frac{12n+12}{5}>2n+2$ but $r<3n+3.$ So, $c_3^{n+1}=0.$ Thus, we have $c_2^{\frac{r}{2}+1}=c_3^{\frac{r}{3}+1}=c_2c_3=c_2^{\frac{r}{2}}+c_3^{\frac{r}{3}}=0,$ where $r\equiv 12k, n\equiv 5k-1;k \in \mathbb{N}.$ This realizes possibility (\ref{5}) for $s=2~\&~q=3.$\\
		If $c_2c_3\neq 0,$ then we have $c_2^{n+1}$ and $c_2^{n}c_3$ are the possible least degree elements. If $c_2^{n+1}=0$ then we must have  $c_2^{n}c_3\neq 0 .$ Clearly, for $c_3^2=0,$ we get $F\sim_2 \mathbb{C}P^n\times \mathbb{S}^3$  and for  $c_3^2\neq 0,$ we must have  $c_3^2=c_2^3.$ Thus, $c_2^{n+1}=c_3^2+c_2^3=c_3^{2l+2}=0$ where $l=[\frac{n}{3}].$  This realizes possibility (\ref{2}) for $q=3.$ \\
		Clearly, $c_2^{n+1}\neq 0$ and $c_2^nc_3\neq 0$ is not possible.\\
		If $c_2^{n+1}\neq 0$ and $c_2^nc_3=0.$ Then, we get $c_2^{\frac{r}{2}}=c_2^{\frac{r-3j}{2}}c_3^j,j>1,$ which generates  $H^r(F)$ and $r$ and $j$ must be even.  Since, rk $H^*(F)=2n+2,$ so we must have $j=2$ and $r=2n+4,$ Hence, the cohomology ring $H^*(F)$ is generated by $c_2$ and $ c_3$ with $c_2^{n+3}=c_2^{3}+c_3^{2}=c_2^{n}c_3=0.$  This realizes possibility (\ref{6}) for $s=2~\&~q=3.$ \\
		If $H^1(F)\neq 0,$ then rk $H^*(F)>2n+2,$ for either $c_2=c_1^2$ or $c_2\neq c_1^2 .$ \\
		\textbf{Subcase(ii):} $c_4 \neq c_2^2$ \\
		In this subcase, we must have $H^1(F)\neq 0.$ It is easy to observed that  rk $H^*(F)>2n+2$ in both case either $c_2^2=0$ or $c_2^2\neq 0.$ \\
		\textbf{Case(6):}  If $B_2=0$ and $B_1=B_3=1,$ then $j^*(\alpha)=1 \otimes c_4+t\otimes c_3+t^3\otimes c_1.$\\
		In this case, if $H^*(F)$ has one generator then $F\sim_2 \mathbb{R}P^{2n+1}.$ Suppose that  $H^*(F)$ has two generators. Now we consider two subcases: (i) $c_3=c_1^3$ (ii) $c_3\neq c_1^3.$\\
		\textbf{Subcase(i):} $c_3=c_1^3.$\\
		First, suppose that  $c_4=c_1^4.$ Then, 
		we have 
		$j^*(\alpha^n) = \sum_{r=0}^{n}\sum_{k=0}^{n-r}{{n}\choose {r}}{{n-r}\choose {k}}t^{3k+r}\otimes c_1^{4n-3r-k}$ and 
		$$\hspace{-4cm}	j^*(\alpha^n \beta )=\sum_{r=0}^{n}\sum_{k=0}^{n-r}\sum_{i=1}^{m-1}{{n}\choose{r}}{{n-r}\choose{r}}t^{3r+r+i}\otimes c_1^{4n-3r+m-i}$$
		$$\hspace{6cm}+ \sum_{r=0}^{n}\sum_{k=0}^{n-r}\sum_{i=q}^{m-1}A_{l,j}{{n}\choose{r}}t^{3r+r+i}\otimes c_1^{4n-3r-k+(m-i-q)}d.$$
		Clearly, $c_1^{n+1}$ and $c_1^nd$ are the possible least degree elements.\\
		If $c_1^{n+1}=0,$ then $c_1^nd$ must be non zero. Thus, for $d^2=0,$ we get $F \sim_2 \mathbb{R}P^n \times \mathbb{S}^q, 1 \leq q\leq m.$ This realizes possibility (1) for $\mathbb{F}=\mathbb{R}.$ And for $d^2\neq 0 ,$ we have two possibility either $d^2=c_1^{2q}$ or $d^2=c_1^qd.$\\
		If $d^2=c_1^{2q},$ then by the change of basis $d'=d+c_1^q,$ we realizes possibility (1) for $\mathbb{F}=\mathbb{R}.$\\
		If $d^2=c_1^qd,$ then for  $1<q\leq n,$  $F$ does not satisfy poincar\'{e} duality. So, we must have $q=1.$ Again, by the change of basis $d'=c+d,$  we get $d'^{n+2}=d^{n+2}=d'^{n+1}+ d'^{n+1} = dd'=0.$ Thus, $F\sim_2 \mathbb{R}P^{n+1}\#\mathbb{R}P^{n+1}.$  This realizes possibility (2) for $\mathbb{F}=\mathbb{R}.$ \\
		Now, If $c_1^{n+1}\neq 0,$ then either $c_1^nd=0$ or $c_1^nd\neq 0..$\\
		Obviously, $c_1^nd\neq 0$ is not possible. Suppose that  $c_1^nd =0.$ If $c_1d=0,$ then we get $c_1^{r+1}=d^{\frac{r}{q}+1}=c_1d=c_1^r+d^{\frac{r}{q}}=0,$ where $r=\frac{q(2n+2)}{q+1}.$ So, $(q+1)|(2n+2),$ and hence, $n=(q+1)k-1$ for $q$ even $\&$ $n=(\frac{q+1}{2})k-1$ for $q$ odd, $k \in \mathbb{N}$. This realizes possibility (\ref{5}) for $s=1.$ Further, If $q=1,$ then $F\sim_2 \mathbb{R}P^{n+1}\# \mathbb{R}P^{n+1},$ if $q=2,$ then $F\sim_2 \mathbb{R}P^{4k}\# \mathbb{R}P^{2k}$ and  if $q=4,$ then $F\sim_2 \mathbb{R}P^{8k}\# \mathbb{H}P^{2k}.$  If $c_1d\neq  0$, then  we get $c_1^r=c_1^{r-qj}d^j,j>1$ which generates  $H^r(F).$  Thus,    $c_1^{qi}\neq d^i$ for $1\leq i \leq j-1$  and $c_1^{qj}= d^j$ where $qj<r.$ We get rk $H^*(F)=j(r-qj)+qj+j$ which must be  $2n+2$ so, $r= \frac{2n+2}{j}+qj-(q+1).$  Hence, the cohomology ring $H^*(F)$ is generated by $c_1$ and $ d$ with $c_1^{r+1}=c_1^{qj}+d^{j}=c_1^{r-qj+1}d=0,$ with $\frac{(q+1)j}{2}-1<n.$ This realizes possibility (\ref{6}) for $s=1.$\\
		Next, suppose that $c_4\neq c_1^4.$\\
		If $c_1^4=0,$ then this case possible only for $n=3,$ when $c_1^nc_4\neq 0~ \&~ c_4^2=0. $ Thus, $F\sim_2 \mathbb{R}P^3 \times \mathbb{S}^4.$ This realizes possibility (1) for $n=3~\&~ q=4.$\\
		If $c_1^4 \neq 0,$ then we have  $j^*(\alpha^n) = \sum_{r=0}^{n}\sum_{k=0}^{n-r}{{n}\choose {r}}{{n-r}\choose {k}}t^{n-r+2k}\otimes c_1^{3n-3r-2k}c_4^{r}$ and  $$j^*(\alpha^n\beta)= \sum_{i=0}^{m-1}\sum_{r=0}^{n}\sum_{k=0}^{n-r}{{n}\choose {r}}{{n-r}\choose {k}}t^{n-r+2k+i} \otimes (\oplus_{l+4j=m-i}c_1^{3n-3r-2k+l}c_4^{r+j}).$$ 
		Note that from above expression we get $c_1^{n+1}$ and $c_1^nc_4$ are the possible least degree elements.
		If  $c_1c_4=0,$ then we get $c_1^r=c_4^{\frac{r}{4}},$ where $r$ is formal dimension. So, $r=8k$ where $n\equiv 5k-1, k \in \mathbb{N}.$ Thus $F\sim_2 \mathbb{R}P^{8k} \# \mathbb{H}P^{2k}.$ This realizes possibility (\ref{5}) for $s=1~\&~ q=4.$ Now, Suppose that  $c_1c_4\neq 0.$  If $c_1^{n+1}=0,$ then we must have  $c_1^{n}c_4\neq 0 .$ Clearly, for  $c_4^2=0,$ we get $F\sim_2 \mathbb{R}P^n\times \mathbb{S}^4$  and for $c_4^2\neq 0,$  we must have  $c_4^2=c_1^8.$ By the change of basis $d'=c_1^4+c_4,$ we  realizes possibility (1) for $\mathbb{F}=\mathbb{R}$ and $q=4.$ \\
		Clearly, $c_1^{n+1}\neq 0$ and $c_1^nc_4\neq 0$ is not possible.\\
		If $c_1^{n+1}\neq 0$ and $c_1^nc_4=0.$ Then, we get $c_1^{r}=c_1^{r-4j}c_4^j,j>1,$ which generates  $H^r(F).$  Thus,    $c_1^{4i}\neq c_4^i$ for $1\leq i \leq j-1$  and $c_1^{4j}= c_4^j$ where $4j<r.$ We get rk $H^*(F)=j(r-4j)+5j$ which must be  $2n+2.$ Thus $r= \frac{2n+2}{j}+4j-5$ and  hence, the cohomology ring $H^*(F)$ is generated by $c_1$ and $ c_4,$ with $c_1^{r+1}=c_4^{4j}+c_4^{j}=c_1^{r-4j+1}c_4=0.$  This realizes possibility (\ref{6}) for $s=1~\&~q=4.$ \\
		\textbf{Subcase (ii):} $c_3\neq c_1^3.$\\
		As $H^*(F)$ has at most two generators so, we must have $c_4=c_1^4.$ Thus $j^*(\alpha^n) = \sum_{r=0}^{n}\sum_{k=0}^{n-r}{{n}\choose {r}}{{n-r}\choose {k}}t^{3k+r}\otimes c_1^{4n-4r-3k}c_3^{r}$ and  $$j^*(\alpha^n\beta)= \sum_{i=0}^{m-1}\sum_{r=0}^{n}\sum_{k=0}^{n-r}{{n}\choose {r}}{{n-r}\choose {k}}t^{3k+r+i} \otimes (\oplus_{l+3j=m-i}c_1^{4n-4r-3k+l}c_3^{r+j}).$$ 
		So, we get  $c_1^{n+1}$ and $c_1^nc_3$ are the possible least degree elements.
		If  $c_1c_3=0,$ then we get $c_1^r=c_3^{\frac{r}{3}},$ where $r$ is formal dimension. So, $r=3k$ and $n\equiv 2k-1, k \in \mathbb{N}.$ Thus, we have $c_1^{3k+1}=c_3^{k+1}=c_1c_3=c_1^{3k}+c_3^{k}, k \in \mathbb{N}.$ This realizes possibility (\ref{5}) for $s=1~\&~ q=3.$ Now, suppose that $c_1c_3\neq 0.$ If $c_1^{n+1}=0,$ then we must have  $c_1^{n}c_3\neq 0 .$ Clearly, for $c_3^2=0,$ we have $F\sim_2 \mathbb{R}P^n\times \mathbb{S}^3$  and for $c_3^2\neq 0,$  we must have  $c_3^2=c_1^6.$ By the change of basis $d'=c_1^3+c_3,$ we  realizes possibility (1) for $\mathbb{F}=\mathbb{R}$ and $q=3.$ \\
		Clearly, $c_1^{n+1}\neq 0$ and $c_1^nc_3\neq 0$ is not possible.\\
		If $c_1^{n+1}\neq 0$ and $c_1^nc_3=0.$ Then, we get $c_1^{r}=c_1^{r-3j}c_3^j,j>1,$ which generates  $H^r(F).$  Thus,    $c_1^{3i}\neq c_3^i$ for $1\leq i \leq j-1$  and $c_1^{3j}= c_3^j$ where $3j<r.$ Clearly, $r= \frac{2n+2}{j}+3j-4,$ where  either $n+1= jk,$ for some $k\in \mathbb{N}$ or $j=2.$     Hence, the cohomology ring $H^*(F)$ is generated by $c_1$ and $ c_3$ with $c_1^{r+1}=c_3^{3j}+c_3^{j}=c_1^{r-3j+1}c_3=0.$  This realizes possibility (\ref{6}) for $s=1~\&~q=3.$ \\
		\textbf{Case(7):}  If $B_3=0$ and $B_1=B_2=1,$ then $j^*(\alpha)=1 \otimes c_4+t^2\otimes c_2+t^3\otimes c_1.$\\
		If $H^*(F)$ has one generator then $F\sim_2 \mathbb{R}P^{2n+1}.$ Suppose that $H^*(F)$ has two generators. We consider two subcases: (i) $c_2=c_1^2$ (ii) $c_2\neq c_1^2.$\\
		\textbf{Subcase (i):}  $c_2=c_1^2.$\\
		First, assume that $c_4=c_1^4.$ We have
		$j^*(\alpha^n) = \sum_{r=0}^{n}\sum_{k=0}^{n-r}{{n}\choose {r}}{{n-r}\choose {k}}t^{2k+3r}\otimes c_1^{4n-3r-2k}$ and 
		$$\hspace{-2cm}	j^*(\alpha^n \beta )= \sum_{r=0}^{n}\sum_{k=0}^{n-r}\sum_{i=1}^{m-1}{{n}\choose{r}}{{n-r}\choose{r}}t^{2k+3r+i}\otimes c_1^{4n-3r-2k+(m-i)}$$
		$$\hspace{4cm}+ \sum_{r=0}^{n}\sum_{k=0}^{n-r}\sum_{i=q}^{m-1}A_{l,j}{{n}\choose{r}}t^{2k+3r+i}\otimes c_1^{4n-3r-2k+(m-i-q)}d.$$ Clearly, $c_1^{n+1}$ and $c_1^nd$ are the possible least degree elements.\\
		If $c_1^{n+1}=0,$ then $c_1^nd$ must be non zero. Thus, for $d^2=0,$ we get $F \sim_2 \mathbb{R}P^n \times \mathbb{S}^q, 1 \leq q\leq m.$ This realizes possibility (1) for $\mathbb{F}=\mathbb{R}.$ And for $d^2\neq 0 ,$ we get $F \sim_2 \mathbb{R}P^n \times \mathbb{S}^q, 1 \leq q\leq m$ and $F\sim_2 \mathbb{R}P^{n+1}\#\mathbb{R}P^{n+1},$ when $d^2=c_1^{2q}$ and $d^2=c_1^qd,$ respectively.\\
		Now, suppose that $c_1^{n+1}\neq 0,$ then either  $c_1^nd=0$  or $c_1^nd\neq 0.$\\
		Clearly, $c_1^nd\neq 0$ is not possible. Now, suppose that  $c_1^nd =0.$ Again, as same above case for  $c_1d=0,$ we have realizes possibility (\ref{5}) for $s=1.$ For $c_1d\neq 0,$  we have realizes possibility (\ref{6}) for $s=1.$ \\
		Next, suppose that  $c_4 \neq c_1^4.$\\
		If $c_1^4=0,$ then for $c_1^3=0,$ we get rk $H^*(F)\neq 2n+2,$ a contradiction. And for $c_1^3\neq 0,$ we must have  $n=3$ and $c_4^2=0.$ So, we have $c_2c_1^2 \neq 0,$ and hence $F\sim_2 \mathbb{R}P^3 \times \mathbb{S}^4.$ Let $c_4 \neq c_1^4\neq 0.$ Then,
		$j^*(\alpha^n) = \sum_{r=0}^{n}\sum_{k=0}^{n-r}{{n}\choose {r}}{{n-r}\choose {k}}t^{3k+r}\otimes c_1^{4n-4r-3k}c_3^{r}$ and  $$j^*(\alpha^n\beta)= \sum_{i=0}^{m-1}\sum_{r=0}^{n}\sum_{k=0}^{n-r}{{n}\choose {r}}{{n-r}\choose {k}}t^{2n+r-2k+i} \otimes (\oplus_{l+4j=m-i}c_1^{2n-r-2k+l}c_4^{r+j}).$$ 
		We get $c_1^{n+1}$ and $c_1^nc_4$ are the possible least degree elements.
		If $c_1c_4=0,$ then we get $c_1^r=c_4^{\frac{r}{4}},$ where $r$ is formal dimension. So, $r=8k$ and $n\equiv 5k-1, k \in \mathbb{N}.$ Thus, $c_1^{8k+1}=c_4^{k+1}=c_1c_4=c_1^{4k}+c_4^{k}, k \in \mathbb{N}.$ This realizes possibility (\ref{5}) for $s=1~\&~ q=4.$ Now, let  $c_1c_4\neq 0.$  If $c_1^{n+1}=0,$ then we must have  $c_1^{n}c_4\neq 0 .$ Clearly, for $c_4^2=0,$ we have $F\sim_2 \mathbb{R}P^n\times \mathbb{S}^3$  and for $c_4^2\neq 0,$ we must have  $c_4^2=c_1^6.$ By the change of basis $d'=c_1^4+c_4,$ we  realizes possibility (1) for $\mathbb{F}=\mathbb{R}$ and $q=4.$ \\
		Clearly, $c_1^{n+1}\neq 0$ and $c_1^nc_4\neq 0$ is not possible.\\
		If $c_1^{n+1}\neq 0$ and $c_1^nc_4=0.$ Then, we get $c_1^{r}=c_1^{r-4j}c_4^j,j>1,$  generates  $H^r(F).$  Thus,    $c_1^{4i}\neq c_4^i$ for $1\leq i \leq j-1$  and $c_1^{4j}= c_4^j,$ where $4j<r.$ Clearly, $r= \frac{2n+2}{j}+4j-5,$ where  either $n+1= jk,$ for some $k\in \mathbb{N}$ or $j=2.$   Hence, the cohomology ring $H^*(F)$ is generated by $c_1$ and $ c_4$ with $c_1^{r+1}=c_4^{4j}+c_4^{j}=c_1^{r-4j+1}c_4=0.$  This realizes possibility (\ref{6}) for $s=1~\&~q=4.$ \\
		\textbf{Subcase(ii):} $c_2\neq c_1^2.$\\
		Suppose that $c_1^2=0.$ Then, we must have  $c_4=c_2^2.$ So, 
		$j^*(\alpha^n) = \sum_{r=0}^{1}{{n}\choose {r}}(1\otimes c_2^2+t^2\otimes c_2)^{n-r}(t^3\otimes c_1)^r.$ Clearly, $c_2^n\neq 0.$ For $c_2^{n+1}=0,$ we get $F \sim_2 \mathbb{C}P^{n}\times \mathbb{S}^1.$ For $c_2^{n+1}\neq0,$ we get rk $H^*(F)>2n+2,$ a contradiction. \\
		Now suppose that $c_1^2\neq 0.$ Since, $H^*(F) $ has at most two generators so, we must have either  $c_4=c_2^2$ or $c_4=c_1^4.$ In both cases, we get $c_1^{n+1}$ and $c_1^nc_2$ are the least degree  element in the image of $j^*(\alpha^n\beta)$. \\
		If $c_1c_2=0,$ then we must have   $c_1^{n+1}\neq 0.$ Thus, $c_1^{r}=c_2^{\frac{r}{2}}$ is the generator of $H^r(F).$ This implies that  rk $H^*(F)={r}+{\frac{r}{2}}=2n+2.$ Thus, $r=4k,$ where $n=3k-1,$  $k \in \mathbb{N}.$ Hence, $F\sim_2 \mathbb{R}P^{4k}\# \mathbb{C}P^{2k}, k \in \mathbb{N}.$ This realizes possibility (\ref{5}) for $s=1~\&~q=2.$\\
		Now, suppose that $c_1c_2 \neq 0.$ If $c_1^{n+1}= 0,$ then $c_1^nc_2$ must be nonzero. By the change of basis $d=c_1^2+c_2.$ we get $F\sim_{2} \mathbb{R}P^n\times \mathbb{S}^2.$ This realizes possibility (\ref{1}) for $\mathbb{F}=\mathbb{R}$ and $q=2.$  \\
		If $c_1^{n+1}\neq 0$ and $c_1^nc_2 \neq 0,$ then rk $H^*(F)>2n+2,$ a contradiction.\\
		If $c_1^{n+1}\neq 0.$ Then, $c_1^{r}=c_1^{r-2j}c_2^j,$  $j>1$ forms  generator of $H^r(F).$ Which implies  that $c_1^{2j}=c_2^{j}$ is generator of $H^{2j}(F).$ Hence, the cohomology ring is given by $c_1^{r+1}=c_1^{2j}+c_2^{j}= c_1^{r-2j+1}c_2=0.$ This realizes possibility (\ref{6}) for $s=1~\&~ q=2.$
		
		\textbf{Case(8):}  If  $B_1=B_2=B_3=1,$ then $j^*(\alpha)=1 \otimes c_4+t \otimes c_3 +t^2\otimes c_2+t^3\otimes c_1.$\\
		If $H^*(F)$ has one generator then $F\sim_2 \mathbb{R}P^{2n+1}.$ Now, suppose that $H^*(F)$ has two generators. We consider two subcases: (i) $c_2=c_1^2$ (ii) $c_2\neq c_1^2.$\\
		\textbf{Subcase (i):} Assume that $c_2=c_1^2.$\\
		If $c_3=c_1^3$ and $c_4=c_1^4,$ then $j^*(\alpha^n)= (1 \otimes c_1^4 + t \otimes c_1^3 +t^2\otimes c_1^2+ t^3 \otimes c_1)^n.$ So, we must have $c_1^n \neq 0.$  This case is similar to Subcase (i) of Case (7), when $c_2=c_1^3~\&~ c_4=c_1^4.$\\ 
		Now, if $c_3=c_1^3$ and $c_4\neq c_1^4,$ then $j^*(\alpha^n) = \sum_{r=0}^{n}{{n}\choose {r}}(1\otimes c_4)^{n-r}(t\otimes c_1^3+t^2\otimes c_1^2+t^3\otimes c_1)^r.$ Further, if $c_1^4=0,$ then this holds only when $n=3$ and $c_4^2=0$ when $c_4^{n-2}c_1^3\neq 0.$ Thus $F\sim_2 \mathbb{R}P^3\times \mathbb{S}^4.$ \\
		If $c_1^4 \neq 0,$ then, $j^*(\alpha^n\beta)=$ $ \sum_{i=0}^{m-1}\sum_{r=0}^{n}\sum_{k=0}^{n-r}\sum_{k'=0}^{r}{{n}\choose {r}}{{n-r}\choose {k}}{{n-r-k}\choose {k'}}t^{n-r+k'+2k+i} \otimes (\oplus_{l+4j=m-i}c_1^{3n-3r-2k-k'+l}c_4^{r+j}).$ Clearly, $c_1^{n+1}$ and $c_1^nc_4$ are the possible least degree elements. This case is similar to  Subcase (i) when $c_4\neq c_1^4$ of  Case (6).\\
		If $c_3 \neq c_1^3,$ then $c_4=c_1^4.$ This case is similar to  Subcase (ii)  of  Case (6).\\
		\textbf{ Subcase (ii):} Assume that $c_2\neq c_1^2.$ We must have either $c_4=c_2^2$ or  $c_4\neq c_2^2.$ This  is similar to Subacse (ii) of case (7).\\
		Finally, if $c_4=0$ or \{$c_4\neq 0 ~\&~B_4=0$\}, then  by Proposition \ref{prop 2.4}, we must have $n=1.$ So, we get  rk $H^*(F)=4$ and   $j^*(\alpha)= B_1 t^3 \otimes c_1 + B_2 t^2 \otimes c_2+B_3 t \otimes c_3$  Clearly, if $H^*(F)$ has one generator then $F\sim_2 \mathbb{F}P^3, \mathbb{F}=\mathbb{R}$ or $\mathbb{C}.$ Suppose  $H^*(F)$ has two  generators then  clearly, $F\sim_{2} \mathbb{S}^r \times \mathbb{S}^q, 1\leq r \leq 3~\&~1 \leq q\leq m, $ or $F\sim_{2} \mathbb{F}P^2\# \mathbb{F}P^2, \mathbb{F}=\mathbb{R}$ or $\mathbb{C}.$ This realizes possibilities (1), (2) and (3) for $\mathbb{F}=\mathbb{R}$ or $\mathbb{C}$ and $n=1.$ 
	\end{proof}
	\begin{remark}
		For $n=1,$ we get $X \sim_2 \mathbb{S}^4 \times \mathbb{S}^m.$ By Theorem \ref{thm 3.3},  the possibilities of connected  fixed point sets  of involutions on $X$ are   $\mathbb{S}^r \times \mathbb{S}^q, 1\leq r \leq 4~\&~1 \leq q \leq m,$ or $\mathbb{F}P^3,\mathbb{F}=\mathbb{R}$ or $\mathbb{C},$ or  $ \mathbb{F}P^2 \# \mathbb{F}P^2, \mathbb{F}=\mathbb{R},$ $\mathbb{C}$ or $\mathbb{H}.$ These possibilities have also been realized in [Theorem 3.11, \cite{j1}].
	\end{remark}	
		
	\begin{remark}
		It is easy to observe that the fixed point sets of involutions on $X\sim_2 \mathbb{H}P^n \times \mathbb{S}^m,$  when  $X$ is not TNHZ in $X_G$ has mod $2$ cohomology of $q$-sphere, where $-1\leq q \leq 4n+m,$ under the assumptions that the associated Lerey-Serre spectral sequence of Borel fibration $ X \hookrightarrow X_G \rightarrow B_G$   is nondegenerate and the differentials $d_r$ of the spectral sequence satisfies $d_r(1\otimes b)=0~ \forall ~r \leq m,$ (See Theorem $3.5$ in \cite{dimpi}).\end{remark}
	
	Now, we give examples to realizes above Theorem.
	\begin{example}
		Let $G=\mathbb{Z}_2$ act  on $\mathbb{S}^m$ defined by $$(x_0,x_1,\cdots x_m) \mapsto (x_0,x_1,\cdots,x_q, -x_{q+1}, \cdots  -x_m) .$$
		If we consider trivial action of  $G$ on $ \mathbb{H}P^n,$   then  after taking diagonal action of $G$ on $\mathbb{H}P^n \times \mathbb{S}^m,$ we get fixed point set is $\mathbb{H}P^n \times \mathbb{S}^q,$ where $ 1\leq q \leq m.$\\
		If we take conjugation action of  $G$ on $ \mathbb{H}P^n, $   i.e $(z_0,z_1,\cdots z_n) \mapsto (\bar{z}_0,\bar{z}_1,\cdots\bar{ z}_n),$ then after taking diagonal action of $G$ on $\mathbb{H}P^n \times \mathbb{S}^m,$ we get fixed point set is $\mathbb{R}P^n \times \mathbb{S}^q,1\leq q \leq m.$ 	\\	 
		If  $G$ acts on  $\mathbb{H}P^n,$ define by $(z_0,z_1,\cdots z_n) \mapsto (i{z}_0,i{z}_1,\cdots i{ z}_n),$ then after taking diagonal action of $G$ on $\mathbb{H}P^n \times \mathbb{S}^m,$ we get fixed point set is $\mathbb{C}P^n \times \mathbb{S}^q,1\leq q \leq m.$ 
		This examples realizes possibility (1) of Theorem  \ref{thm 3.3}.\\
		Now, consider an action of $G$ on $\mathbb{S}^4$ defined by $(x_0,x_1,x_2,x_3, x_4) \mapsto (x_0,x_1,x_2,x_3,-x_4) .$ Then, the fixed point set of the diagonal action of $G$ on $\mathbb{S}^4\times \mathbb{S}^m$ is $\mathbb{S}^3\times \mathbb{S}^q, 1\leq q \leq m.$ This also realizes possibility (1) of Theorem \ref{thm 3.3} for $n=1.$
		
	\end{example}
	\begin{example}
		Bredon (\cite{Bredon}) constructed an example that  $\mathbb{P}^{2}(q)\# \mathbb{P}^2(q)$ (connected sum of projective spaces) is a fixed point set of an involution on $\mathbb{S}^4 \times 
		\mathbb{S}^{q+k},$ where  $k\geq 4.$ This example   realizes possibility (2) of Theorem \ref{thm 3.3} for $n=1.$ In this paper, Bredon also gave  examples of involutions on $X\sim_2 \mathbb{S}^n \times \mathbb{S}^m,n\leq m$ and $X\sim_2 \mathbb{S}^4 \times \mathbb{S}^m,4< m$ with the fixed point sets $F=\mathbb{R}P^3$ and   $F\sim_2 \mathbb{S}^{7},$ respectively.  These examples   realizes possibility (3) of Theorem  \ref{thm 3.3} for $n=1,$ and the case when $X\sim_2 \mathbb{H}P^1\times \mathbb{S}^m$  and  not TNHZ  in $X_G,$  respectively.
		
	\end{example}

	Next, we have discussed the cohomology ring of the orbit spaces of free involutions on a space $X$ having mod $2$ cohomology the  product of  quaternionic projective space and sphere $ \mathbb{H}P^n \times \mathbb{S}^m.$ For the existence of free involutions on $ \mathbb{H}P^n \times \mathbb{S}^m,$ consider the diagonal action on $\mathbb{H}P^n\times \mathbb{S}^m,$ by taking any involution on 	$\mathbb{H}P^n$ and antipodal action on $\mathbb{S}^m.$ \\
First, we have consider the case when $\pi_1(B_G)$ acts trivially on $H^*(X)$ under some assumptions on the associated Lerey-Serre spectral sequence of Borel fibration $ X \hookrightarrow X_G \rightarrow B_G.$ 
	 Note that \cite{dimpi} if $G=\mathbb{Z}_2$ act freely on $X\sim_2 \mathbb{H}P^n\times \mathbb{S}^m,$ then $\pi_1(B_G)$  acts trivially on $H^*(X)$ whenever one of the following holds: (1) $ 4n \leq m$ (2) $ 4=m<4n, n$ is even, (3) $4<m<2m\leq 4n, m\equiv 0$(mod $4$), and (4) $m\not \equiv 0$(mod $4$).
	\begin{theorem}\label{thm 4.3}
	Let $G=\mathbb{Z}_2$ act freely on finite CW-complex $X\sim_2\mathbb{H}P^n\times \mathbb{S}^m,$ where  $n,m\geq 1.$ Assume that $\pi_1(B_G)$ acts trivially on $H^*(X)$ and the differentials $d_r(1\otimes b)=0~\forall~ r\leq m.$ Then, the  cohomology ring of orbit space $H^*(X/G)$ is isomorphic to one of the following  graded commutative algebras: 	\begin{enumerate}
		\item $\mathbb{Z}_2[x,y,z]/I,$ where $I$ is homogeneous ideal given by:
		\begin{center}
			$<x^{5},y^{\frac{n+1}{2}}+a_0y^{\frac{4(n+1)-m}{8}}z+a_1x^{4}y^{\frac{4n-m}{8}}z+a_2z,z^{2}+a_3x^{2i}y^{\frac{m-i}{4}}+a_{4}x^{i'}y^{\frac{m-{i'}}{8}}z>,$ 
		\end{center}
		where  deg $x=1,$ deg $y=8~\&$ deg $z=m,$ $a_0=0$ if $m \not\equiv 0 $(mod $8$) or $m>4n+4;$ $a_1=0$ if $m \equiv 0 $(mod $8$) or $m>4n;$ $a_2=0$ if $m \neq 4(n+1);$
		$a_3=0$ if $m \not \equiv i$(mod $4$) or  \{$i=0$ and $2m > 4(n-1)\}, 0\leq 2i\leq 4$ and
		$ a_{4}=0$ if $m \not\equiv {i'} $(mod $8$) or $m>4n, 0\leq {i'}\leq 4,$  
		$a_k\in \mathbb{Z}_2, 0\leq k\leq 4,$  $n$ odd,	\vspace{0.5em}
		\item $\mathbb{Z}_2[x,y,z]/<x^{5},y^{\frac{n}{2}+1},z^2+a_0y+a_1x^4z>,$ where  deg $x=1,$ deg $y=8~\&$ deg $z=4,$ $a_0,a_1\in \mathbb{Z}_2,$ $n$  even, and

		\vspace{0.5em}
		\item  $\mathbb{Z}_2[x,y]/<x^{m+1},y^{n+1}+{\sum_{0<i\equiv 0 (mod ~4)}^{min \{4(n+1),m\}}}a_ix^{i}y^{\frac{4(n+1)-i}{4}}>,$ 
		where deg $x=1,$ deg $y=4$ and $a_i\in \mathbb{Z}_2.$   
	\end{enumerate}
	
\end{theorem}

	Finally, we have consider the case  when $\pi_1(B_G)$ acts nontrivally on $H^*(X).$ \begin{theorem}\label{thm 4.2}
		Let $G=\mathbb{Z}_2$ act freely on a finite CW-complex $X\sim_2\mathbb{H}P^n\times \mathbb{S}^m,$ where $ n, m\geq 1.$ Assume that $\pi_1(B_G)$ acts nontrivially on $H^*(X).$ Then, $H^*(X/G)$ is isomorphic to one of the following graded commutative  algebras:
		\begin{enumerate}
			\item ${\mathbb{Z}_2[x,y,z]}/{<x^{9},y^2+a_0z+a_1x^{8},z^{\frac{n+1}{2}}+a_2x^{8}z^{\frac{n-1}{2}},xy>,}$
			where  deg $x=1,$ deg $y=4,\&~$ deg $z=8, a_i\in \mathbb{Z}_2, 0\leq i \leq 2,m=4<4n,$  $n$ odd, and 
			\item ${\mathbb{Z}_2[x,y,z,w_{k}]}/{<x^{5},y^{\frac{m}{8}}+a_0w_1,z^{2},xw_{k},{w_kw_{k+i}}+a_{k,i}x^{4d}y^{\frac{2m-4n+4q}{8}}z>}, ~$
			where$~$ deg $x=1,$ deg $y=8,$ deg $z=4n+4~\&$ deg $w_{k}=m+4(k-1),$ $1\leq k\leq \frac{4n-m+4}{4},$ and  $0\leq i\leq \frac{4n-m}{4},$  $-1\leq q(odd)\leq \frac{4n-m-8}{4} ,d=0,1,~\&~ $  $n$  odd, $4<m<4n<2m,m\equiv 0$ (mod $8$) and $a_{k,i}=0$ if $\frac{4(n+2)-m}{4}<2k+i;a_0$ and ${a_{k,i}}'s$ are in $\mathbb{Z}_2.$
			If $d=0,$ then $i$ is even and $q=2k+i-3.$ If $d=1,$ then $i$ is odd and $q=2k+i-4.$ 
		\end{enumerate}
		\end{theorem}
The proofs of the above Theorems are similar to proofs of Theorem $4.2$ and Theorem $4.5$ in \cite{dimpi}, respectively.

	\begin{remark}
	If  $a_i=0, 0\leq i \leq 2$ in possibility (1) of the Theorem \ref{thm 4.2}, then $X/G\sim_2  (\mathbb{R}P^{8} \vee \mathbb{S}^4) \times\mathbb{P}^{\frac{n-1}{2}}(8).$  %We have  used the notation $ \mathbb{P}^{\frac{n-1}{2}}(2l) $ for the projective spaces. 
		If $a_i=0~\forall~ 0\leq i \leq 4,$ in possibility (1) of Theorem \ref{thm 4.3}, then $X/G\sim_2 \mathbb{R}P^4\times \mathbb{P}^{\frac{j}{2}}(8) \times \mathbb{S}^m, j=n-1$ for $n$ odd and $j=n$ for $n$ even. If $a_i=0~ \forall~ 0<i\equiv 0(mod~4)\leq min \{4(n+1),m\}$ in the possibility (3), then $X/G\sim_2 \mathbb{R}P^m \times \mathbb{H}P^n.$
	\end{remark}
	\begin{example}
	Let $T: \mathbb{H}P^n \times  \mathbb{S}^m \rightarrow \mathbb{H}P^n \times  \mathbb{S}^m $ be a map defined by $([z],x)\mapsto ([z],-x).$ Then, this gives a free involution on $\mathbb{H}P^n \times  \mathbb{S}^m.$  The orbit space of this action is $(\mathbb{H}P^n \times  \mathbb{S}^m)/\mathbb{Z}_2 \sim_2 \mathbb{H}P^n \times  \mathbb{R}P^m.$  This realizes possibility (3) of Theorem \ref{thm 4.3}, for $a_i=0~\forall ~i.$
	\end{example}

	\bibliographystyle{plain}

\begin{thebibliography}{100}
		\bibitem{bredon}   G. E. Bredon, Introduction to Compact Transformation Groups. New York, USA Academic Press, (1972).
		
		
		\bibitem{Bredon}  G. E.  Bredon, The cohomology ring structure of a fixed point set.  Ann. of Math. \textbf{80}, 524-537(1964).
		\bibitem{G.Bredon}  G. E.  Bredon,  Cohomological aspects of transformation groups. Proceedings of the Conference on Transformation Groups (New Oreleans, 1967), Springer-Verlag, New York, 245-280(1968).
		
		\bibitem{chang} C. N. Chang and  J. C. Su,  Group actions on a product of two projective spaces, Amer. J. Math., \textbf{101(5)}, 1063-1081(1979).
	
	%	\bibitem{c} Chang, T.,  Comenetz,  M.: Group actions on cohomology projective planes and products of spheres.  Q. J. Math. \textbf{28}, (1977). 
	%	\bibitem{chang} Chang, C. N., Su, J. C.: Group actions on a product of two projective spaces. Amer. J. Math., \textbf{101(5)}, 1063-1081(1979).
	%	\bibitem{co} Coelho, F. R. C., Mattos,  D. de,  Santos,  E. L. dos:  On the existence of G-equivariant maps.  Bull. Braz. Math. Soc. \textbf{43}, 407-421(2012).
	%	\bibitem{floyd}  Conner, P. E.,  Floyd, E. E.:  Fixed point free involutions and equivariant maps.  Bull. Amer. Math. Soc. \textbf{66}, 416-441(1960).
	%	\bibitem{orbit space}  Dotzel, R. M.,  Singh, T. B.,  Tripathi, S. P.:  The cohomology rings of the orbit spaces of free transformation groups of the product of two spheres.   Proc. Amer. Math. Soc. \textbf{129}, 921-930(2000).
		\bibitem{dey}  P. Dey and   M. Singh,  Free actions of some compact groups on Milnor manifolds. Glasg. Math. J. \textbf{61}, 727-742(2019).
		\bibitem{dimpi}Dimpi and H. K. Singh, Involution in the product of projective space and sphere, arXiv:2303.16478v1,(2023).
	%	\bibitem{allen} Hatcher,  A.: Algebraic Topology. Cambridge University Press, Cambridge, (2000). 
		%	\bibitem{monika} Jain,  M.: A study of homology of fixed point set. M.Phil Dissertation, University of Delhi, (1997)
	%	\bibitem{anju} A. Kumari and   H. K. Singh,    Fixed point free actions of spheres and equivarient maps.   Topology Appl. \textbf{305}, 107886(2022). 
		%	\bibitem{kaur}  Kaur, J.,   Singh, H. K.:  On the existence of free action of $\mathbb{S}^3$ on certain finitistic mod p cohomology spaces.  J. Indian Math. Soc. \textbf{82},  97–106(2015)
		%\bibitem{mac}   McCleary, J.: A user’s guide to spectral sequences. Cambridge Studies in Advanced Mathematics, 
	%	Cambridge University Press, IInd edition, \textbf{58}, (2001).
	%	\bibitem{mort}Mostert, P.S. ed.:  Proceedings of the Conference on Transformation Groups. Springer Science \& Business Media.
		\bibitem{Mattos} A. M. M.  Morita,   D. De Mattos and   P. L. Q. Pergher, The cohomology ring of orbit spaces of free $\mathbb{Z}_2$-actions on some 
		Dold manifolds. Bull. Aust. Math. Soc. \textbf{97}, 340-348(2018).
%		\bibitem{hsingh} Singh, H. K., Singh, T. B.:  Fixed point free involutions on cohomology
%		projective spaces, Indian J. Pure Appl. Math., \textbf{39}, 285-291(2008).
	%	\bibitem{nontrivial} Singh, H. K., Singh, K. S.: Indices of a finitistic space with mod $2$ cohomology $\mathbb{R}P^n \times \mathbb{S}^2,$ Indian J. Pure Appl. Math., \textbf{50}(1), 23-34(2019).
	\bibitem{peltier} C. F. Peltier and  R. P.  Beem, Involutions on DOLD manifolds, Proc. Amer. Math. Soc. \textbf{85}, 457-460(1982).
	
		\bibitem{msingh}  M. Singh, Orbit spaces of free involutions on the product of two projective spaces, Results Math.  \textbf{57(1)}, 53-67(2010).
		\bibitem{steenrold}   N. E. Steenrod, Cohomology operations. Annals of Mathematics Studies, No. 50 Princeton University Press, Princeton, N.J. (1962).
		\bibitem{s}    P. A. Smith, Fixed-point theorems for periodic transformations. Amer. J. Math. \textbf{63}, 1-8(1941).
		\bibitem{smith}  P. A.  Smith,  New results and old problems in finite transformation groups, Bull. Amer. Math. Soc. \textbf{66}, 401-415(1960).
		
		%	\bibitem {hemant sir} Pergher, P. L. Q.,   Singh,  H. K.,  Singh, T. B.: On $\mathbb{Z}_2$ and $S^1$ free actions on spaces of cohomology type (a,b).   Houston J. Math. \textbf{36(1)}, 137-146(2010)
		%	\bibitem{singh} Singh,  M.: $\mathbb{Z}_2$ actions on complexes with three non-trivial cells.  Topology Appl. \textbf{155}, 965-971(2008)
		%	\bibitem{m circle} Singh,  M.:  Fixed points of circle actions on spaces with rational cohomology of $S^ n \vee
		%	S^{2n} \vee S ^{3n}$ or $P ^2 (n) \vee S ^{3n}.$  Arch. Math. \textbf{92}, 174-183(2009)
	%	\bibitem{j2}     J. C. Su, Transformation groups on cohomology projective spaces, Trans. Amer. Math. Soc. \textbf{106}, 305-318 (1963).
		
		\bibitem{j1}  J. C. Su, Periodic transformations on the product of two spheres,  Trans. Amer.
		Math. Soc. \textbf{112}, 369-380(1964).
		%	\bibitem{tom}  Dieck, T. Tom:  Transformation Groups. De Gruyter Studies in Mathematics,  \textbf{8}, Walter de Gruyter \& Co., Berlin, (1987)
		\bibitem{puppe}  P. Volker, On a conjecture of Bredon, Manuscripta mathematica, \textbf{12}, 11-16(1974).
			\end{thebibliography}

\end{document}